\documentclass[a4paper,11pt,reqno,noindent]{amsart}
\usepackage[centertags]{amsmath}
\usepackage{amsfonts,amssymb,amsthm,dsfont,cases,amscd,esint,enumerate, stmaryrd}
\usepackage[T1]{fontenc}
\usepackage[english]{babel}
\usepackage[applemac]{inputenc}
\usepackage{newlfont}
\usepackage{color}
\usepackage[body={17cm,21.5cm}, top = 1in, bottom = 1.5in, centering]{geometry} 
\usepackage{fancyhdr}
\usepackage{enumerate}
\usepackage{comment}

\theoremstyle{plain}
\newtheorem{theor}{Theorem}[section]
\newtheorem{lem}[theor]{Lemma}
\newtheorem{prop}[theor]{Proposition}

\newtheorem{rem}[theor]{Remark}
\theoremstyle{definition}

\mathchardef\emptyset="001F
\numberwithin{equation}{section}

\newcommand{\R}{\mathbb R}

\newcommand{\T}{\mathbb T}
\newcommand{\Dd}{{\mathbb D}}
\newcommand{\e}{\varepsilon}
\newcommand{\Pc}{\mathcal{P}}

\newcommand{\Id}{\operatorname{Id}}

\newcommand{\Tr}{\operatorname{Tr}}
\newcommand{\loc}{{\operatorname{loc}}}

\usepackage[colorlinks,citecolor=black,urlcolor=black]{hyperref}
\usepackage{tikz}
\usetikzlibrary{fit}
\usetikzlibrary{shapes.geometric}
\usetikzlibrary{calc}

\usepackage[colorlinks,citecolor=black,urlcolor=black]{hyperref}

\title[Singular mean-field limits for fluctuations around equilibrium]{Singular mean-field limits for fluctuations\\around equilibrium}

\author[M.~Duerinckx]{Mitia Duerinckx}
\address{\sc M. Duerinckx. Universit\'e Libre de Bruxelles, D\'epartement de Ma\-th\'e\-matiques, Brussels, B-1050, Belgium}
\email{mitia.duerinckx@ulb.be}
\author[P.-E. Jabin]{Pierre-Emmanuel Jabin}
\address[Pierre-Emmanuel Jabin]{Penn State University, Department of Mathematics, State College, PA 16802, USA}
\email{pejabin@psu.edu}

\begin{document}

\begin{abstract}
This work addresses the mean-field limit of inertial particle systems with singular interactions in a perturbative regime around Gibbs equilibrium. We prove that small fluctuations around equilibrium are asymptotically governed by the linearized Vlasov equation. The result applies to a broad class of singular interaction kernels, including the Coulomb case in dimensions $d\le3$. In particular, this provides a rigorous derivation of the linearized mean-field dynamics near equilibrium in settings where the corresponding nonlinear mean-field limit remains out of reach.
\end{abstract}

\maketitle
\setcounter{tocdepth}{1}
\tableofcontents
\allowdisplaybreaks

\section{Introduction}
The derivation of the Vlasov equation as a mean-field limit for inertial particle systems with singular interactions remains a major open problem beyond the case of $L^2$ interaction forces recently treated in~\cite{BDJ-24}. In particular, even the two-dimensional Coulomb case remains open.

In this work, we address a perturbative version of this problem around equilibrium. More precisely, we study the evolution of small fluctuations around the $N$-particle Gibbs equilibrium and prove that, for a large class of singular interactions including the Coulomb case in dimensions $d\le3$, the limiting dynamics is governed by the linearized Vlasov equation.
We thus identify the first variation of the dynamics around equilibrium in situations where the nonlinear mean-field limit is still out of reach.

\medskip
Consider a system of $N$ exchangeable particles evolving according to Newtonian dynamics with pairwise interactions. For simplicity, we work on the periodic torus $\T^d$ in dimension $d\ge1$; the extension to the whole space with a confining potential is discussed in Appendix~\ref{app:Rd}. Denoting by $(X_{N,j},V_{N,j})$ the position and velocity of particle $j$ in phase space $\Dd:=\T^d\times\R^d$, the dynamics reads
\begin{equation*}
\left\{\begin{array}{l}
\tfrac{d}{dt}X_{N,j}=V_{N,j},\\[1mm]
\tfrac{d}{dt}V_{N,j}=\frac1N\sum_{l:l\ne j}K(X_{N,j}-X_{N,l}),\qquad 1\le j\le N,
\end{array}\right.
\end{equation*}
where $K=-\nabla W$ for some even potential $W:\T^d\to\R$.
At the level of the $N$-particle distribution function $F_N$ on $\Dd^N$, the dynamics formally leads to the Liouville equation
\begin{equation}\label{eq:Liouville}
\partial_tF_N+\sum_{j=1}^Nv_j\cdot\nabla_{x_j}F_N+\frac1N\sum_{j\ne l}^NK(x_j-x_l)\cdot\nabla_{v_j}F_N=0.
\end{equation}
We shall study this equation close to its Gibbs equilibrium
\[M_{N,\beta}(z_1,\ldots,z_N):=Z_{N,\beta}^{-1}\exp\bigg(-\frac\beta2\sum_{j=1}^N|v_j|^2-\frac\beta{2N}\sum_{j\ne l}^NW(x_j-x_l)\bigg),\]
where $\beta>0$ is the inverse temperature and $z_j=(x_j,v_j)$. The one-particle marginal of this equilibrium is the Maxwellian
\[M_\beta(z):=Z_\beta^{-1}e^{-\frac\beta2|v|^2}.\]
Motivated by chaotic initial data, one may start from initial densities of the form
\[M_{N,\beta}(1+\delta f_\circ)^{\otimes N},\]
for some bounded perturbation $f_\circ:\Dd\to\R$. In the regime of vanishing fluctuations, $\delta\ll N^{-1}$, expanding the product, subtracting the equilibrium, and rescaling by $\delta^{-1}$, this leads to considering (signed) initial data of the form
\begin{equation}\label{eq:initial}
F_N|_{t=0}=M_{N,\beta}\sum_{j=1}^Nf_\circ(z_j),\qquad f_\circ\in L^\infty(\Dd),\qquad \int_\Dd f_\circ M_\beta=0.
\end{equation}
In the sequel, we consider arbitrary bounded weak solutions $F_N$ of~\eqref{eq:Liouville} with initial data~\eqref{eq:initial} satisfying the natural conservation estimate
\begin{equation}\label{eq:conserv}
\int_{\Dd^N}\frac{|F_N(t)|^q}{|M_{N,\beta}|^{q-1}}\le\int_{\Dd^N}\frac{|F_N^\circ|^q}{|M_{N,\beta}|^{q-1}}\quad\text{for all $t\ge0$ and $1\le q<\infty$.}
\end{equation}
This estimate is the weak form of the formal conservation of the $L^q(M_{N,\beta})$-norm of $F_N/M_{N,\beta}$ along the Hamiltonian flow. Such bounded weak solutions exist globally, for instance, whenever $K\in L^1(\T^d)^d$ and $f_\circ\in L^\infty(\Dd)$.

\medskip
Our main result identifies the limiting dynamics of these fluctuation densities: under suitable assumptions on the interaction kernel, the marginals of~$F_N$ are asymptotically governed by the linearized Vlasov equation. In particular, the result covers the 2D Coulomb case at arbitrary temperature, and the 3D Coulomb case at sufficiently high temperatures.

\begin{theor}\label{th:main}
Let the interaction kernel be given by $K=-\nabla W$, for some even potential $W\in L^{2,\infty}(\T^d)$ with bounded negative part $W_-\in L^\infty(\T^d)$, and assume $K\in L^p(\T^d)$ for some~$p>\frac{4}{3}$.
Then there exists~$\beta_0>0$ such that the following holds for all $\beta\in(0,\beta_0)$.
In the special case of the repulsive logarithmic kernel $\hat W(\xi)=|\xi|^{-d}\mathds1_{\xi\ne0}$, one can take $\beta_0=\infty$.

Let $F_N$ be a global bounded weak solution to~\eqref{eq:Liouville} with initial data~\eqref{eq:initial} such that~\eqref{eq:conserv} holds. For every $m\ge1$, the $m$-particle marginal
\begin{equation}\label{eq:marginalFNm}
F_{N,m}(t,z_1,\ldots,z_m)
=\int_{\Dd^{N-m}}F_N(t,z_1,\ldots,z_N)\,dz_{m+1}\cdots dz_N
\end{equation}
converges, as $N\uparrow\infty$, in the sense of distributions on $\R^+\times\Dd^m$ to
\[
\sum_{j=1}^m f(t,z_j)\,\prod_{l:l\ne j}^mM_\beta(z_l),
\]
where $f$ is the unique weak solution of the linearized Vlasov equation
\begin{equation}\label{eq:Vlasov}
\left\{
\begin{array}{l}
\partial_tf+v\cdot\nabla_x f+(K\ast f)\cdot\nabla_v M_\beta=0,\\
f|_{t=0}=f_\circ.
\end{array}
\right.
\end{equation}
\end{theor}

\begin{rem}[High-temperature restriction]
As explained in the proof, the above result extends to arbitrary inverse temperatures~$\beta<\infty$ provided that the spatial Gibbs partition function
\[\bar Z_{N,\beta}=\frac1{|\T^d|^N}\int_{(\T^d)^N}\exp\Big(-\frac{\beta}{2N}\sum_{j\ne l}^N W(x_j-x_l)\Big)\,dx_1\cdots dx_N\]
remains uniformly bounded in the mean-field limit, that is, after the normalization $\int_{\T^d}W=0$,
\begin{equation}\label{eq:bound-Z}
\limsup_{N\uparrow\infty} \bar Z_{N,\beta}<\infty.
\end{equation}

In Theorem~\ref{th:partition}(i) below, we prove that this property~\eqref{eq:bound-Z} holds for sufficiently small~$\beta$ whenever $W\in L^2(\T^d)$ and $W_-\in L^\infty(\T^d)$. Conversely, Corollary~\ref{cor:equ-fl} shows that the condition $W\in L^2(\T^d)$ is necessary for~\eqref{eq:bound-Z}. It also shows that the smallness restriction on~$\beta$ cannot in general be removed either unless the potential is $H$-stable, in the sense that $\widehat{W}\ge0$.

This leaves open the natural question whether~\eqref{eq:bound-Z} holds for all $\beta<\infty$ for every $H$-stable potential $W\in L^2(\T^d)$. To our knowledge, this is not known in general. The case of repulsive logarithmic interactions was recently settled in~\cite{2506.22083}. For more general repulsive Riesz interactions, the best available bounds~\cite{Rougerie-Serfaty,Petrache-Serfaty} yield only $\log \bar Z_{N,\beta}=O(N^{s/d})$, where $s\in(0,d)$ is the Riesz exponent. Such diverging bounds are expected to be optimal only in the much lower-temperature regime $\beta=O(N)$, cf.~\cite{Rosenzweig-pers}.
\end{rem}

\section{Partition function estimates}\label{sec:partition}
For convenience, assume that the torus $\T^d$ has unit volume and that the interaction potential $W$ is centered, $\int_{\T^d}W=0$. We consider the spatial Gibbs partition function
\begin{equation}\label{eq:partition-def}
\bar Z_{N,\beta}:=\int_{(\T^d)^N}
\exp\Big(-\frac{\beta}{2N}\sum_{i\neq j}^NW(x_i-x_j)\Big)\,dx_1\ldots dx_N.
\end{equation}
The purpose of this section is to establish bounds on $\bar Z_{N,\beta}$ in the mean-field limit $N\uparrow\infty$. Such bounds will play a central role in the sequel.
Indeed, the weighted $L^p$-norm of the density of the interacting Gibbs equilibrium $M_{N,\beta}$ with respect to the product mean-field equilibrium $M_\beta^{\otimes N}$ is given by
\begin{equation}\label{eq:moment-MN}
\int_{\Dd^N}\frac{|M_{N,\beta}|^{p}}{|M_\beta^{\otimes N}|^{p-1}}=\frac{\bar Z_{N,p\beta}}{(\bar Z_{N,\beta})^p},
\end{equation}
see also~\cite{DSR-21}. Uniform control of $\bar Z_{N,\beta}$ therefore prevents the Gibbs equilibrium from concentrating on configurations that are atypical under the product mean-field measure, thus providing a strong form of statistical stability for the equilibrium mean-field approximation.

We prove below a collection of high-temperature estimates. In the case $W\in L^2(\T^d)$, we obtain a uniform bound on the partition function. In the critical case $W\in L^{2,\infty}(\T^d)$, we obtain a polynomial bound in $N$. For more singular Riesz-type interactions, we obtain stretched-exponential bounds.
As will be explained later, these estimates are expected to be sharp.

During the completion of this work, we learned that the uniform estimate for $L^2$ interactions had been obtained independently by M. Rosenzweig~\cite{Rosenzweig-pers}, by a different argument based on concentration estimates for canonical $U$-statistics, in particular on~\cite[Theorem~3.3]{Gine-Latala-Zinn-00}. Our proof is instead based on a direct cluster-expansion argument. The bounds for $W\notin L^2(\T^d)$ are then deduced by interpolation.

\begin{theor}[High-temperature bound]\label{th:partition}
Let $W\in L^{1}(\T^d)$ be even and centered, $\int_{\T^d}W=0$, with $W_-\in L^\infty(\T^d)$, and assume that
\[\beta\Big(\|W\|_{L^1(\T^d)}+\|W_-\|_{L^\infty(\T^d)}\Big)\ll1\]
is smaller than some universal constant. Then the spatial Gibbs partition function~\eqref{eq:partition-def} satisfies the following estimates.
\begin{enumerate}[(i)]
\item If $W\in L^2(\T^d)$, then for all $N\ge1$,
\[1\le \bar Z_{N,\beta}\le \exp(C\beta^2\|W\|_{L^2(\T^d)}^2).\]
\item If $W\in L^{2,\infty}(\T^d)$, then for all $N\ge1$,
\[1\le\bar Z_{N,\beta}\le CN^{C\beta^2\|W\|_{L^{2,\infty}(\T^d)}^2}.\]
\item If $W$ is the repulsive periodic Riesz kernel of order $0<s<d$, that is,
\[\widehat W(\xi)=|\xi|^{s-d}\mathds1_{\xi\ne0},\]
then for all $N\ge1$,
\begin{equation}\label{eq:Riesz-cas-partition}
1\le\bar Z_{N,\beta}\le \left\{\begin{array}{lll}
C&:&s<\frac d2,\\
CN^{C\beta^2}&:&s=\frac d2,\\
C\exp(C\beta^{d/s}N^{2-d/s})&:&s>\frac d2.
\end{array}\right.
\end{equation}
\end{enumerate}
Here the constants $C$ depend only on $d$ in~(i)--(ii), and only on $d,s$ in~(iii).
\end{theor}

We next show that the condition $W\in L^2(\T^d)$ is sharp for the uniform boundedness of the partition function. More precisely, when $W\in L^2(\T^d)$, the partition function has an explicit limit as $N\uparrow\infty$, and the necessity of the $L^2$-condition then follows directly from the divergence of the latter. This also shows that, in general, a restriction on $\beta$ cannot be avoided, unless the potential is $H$-stable, in the sense that $\widehat W\ge0$. Finally, for Riesz kernels, the estimates of Theorem~\ref{th:partition}(iii) are consistent with this limiting formula: introducing a natural Fourier cut-off at frequencies $|\xi|\lesssim N^{1/s}$ in~\eqref{eq:limZ} yields precisely the divergences displayed in~\eqref{eq:Riesz-cas-partition}, thereby supporting the expected optimality of those bounds.

\begin{prop}[Equilibrium fluctuations]\label{cor:equ-fl}
Let $W\in L^1(\T^d)$ be even and centered, $\int_{\T^d}W=0$, with $W_-\in L^\infty(\T^d)$.
\begin{enumerate}[(i)]
\item If $\beta\|W\|_{L^1(\T^d)}\ll1$ is smaller than some universal constant and if $W\in L^2(\T^d)$, then
\begin{equation}\label{eq:limZ}
\lim_{N\uparrow\infty}\bar Z_{N,\beta}
=\exp\bigg(\sum_{\xi\ne0}\Big(\tfrac\beta2\widehat W(\xi)-\log(1+\tfrac\beta2\widehat W(\xi))\Big)\bigg)<\infty.
\end{equation}
\item If $W\notin L^2(\T^d)$, then for all $0<\beta<\infty$,
\[\lim_{N\uparrow\infty}\bar Z_{N,\beta}=\infty.\]
\end{enumerate}
\end{prop}

The proof of Theorem~\ref{th:partition} is split into Sections~\ref{sec:L2case}--\ref{sec:noL2case}, while the proof of Proposition~\ref{cor:equ-fl} is postponed to Section~\ref{sec:divergence}.

\subsection{Partition function estimate for $L^2$ interactions}\label{sec:L2case}
This section is devoted to the proof of Theorem~\ref{th:partition}(i).
The lower bound $\bar Z_{N,\beta}\ge1$ follows from Jensen's inequality with $\int_{\T^d}W=0$, so it remains to prove the upper bound.
To this end,
we appeal to cluster expansion techniques, see e.g.~\cite{Brydges,Fernandez-Procacci,Jansen}.
Let $W\in L^2(\T^d)$ be an even periodic potential with $W_-\in L^\infty(\T^d)$ and $\int_{\T^d}W=0$.
In terms of
\[f(x):=e^{-\frac\beta NW(x)}-1,\qquad c_0:=\int_{\T^d}f,\qquad h=\frac{f-c_0}{1+c_0},\]
we can rewrite the partition function~\eqref{eq:partition-def} as
\begin{equation}\label{eq:def-Zh}
\bar Z_{N,\beta}
=(1+c_0)^{\binom{N}{2}}Z_h,\qquad Z_h:=\int_{(\T^d)^N}\prod_{i<j}^N\Big(1+h(x_i-x_j)\Big)dx_1\ldots dx_N.
\end{equation}
Using $\int_{\T^d}W=0$, Jensen's inequality and Taylor's formula give
\begin{equation}\label{eq:est-c0}
0\le c_0=\int_{\T^d}e^{-\frac\beta NW(x)}-1\le \frac{\beta^2}{2N^2}\|W\|_{L^2(\T^d)}^2e^{\frac\beta N\|W_-\|_{L^\infty}},
\end{equation}
hence
\[(1+c_0)^{\binom N2}\le \exp\Big(\beta^2\|W\|_{L^2(\T^d)}^2e^{\frac\beta N\|W_-\|_{L^\infty}}\Big).\]
It remains to estimate $Z_h$ in~\eqref{eq:def-Zh}.
For this purpose, we appeal the following Mayer expansion, see e.g.~\cite[Theorem~4.12]{Jansen},
\begin{equation}\label{eq:mayer-Z}
\log Z_h
=\sum_{\substack{S\subset\{1,\ldots,N\}\\ |S|\ge2}}
\varphi(S),\qquad\varphi(S):=\sum_{\substack{G\text{ connected}\\ V(G)=S}}
\int_{(\T^d)^{S}}\prod_{\{i,j\}\in E(G)}h(x_i-x_j)\,dx_S,
\end{equation}
where the last sum runs over connected graphs $G$ with vertex set $V(G)=S$ and where $E(G)$ stands for the  edge set of $G$.
In order to estimate these sums over graphs, we appeal to the Penrose tree-graph identity, see e.g.~\cite[Section~4.1]{Fernandez-Procacci} or~\cite[(7.1)]{Jansen}: using the short-hand notation $h_e=h(x_i-x_j)$ for an edge~$e=\{i,j\}$,
\begin{equation*}
\sum_{\substack{G\text{ connected}\\ V(G)=S}}~\prod_{e\in E(G)} h_e
=\sum_{\substack{T\text{ tree}\\ V(T)=S}}\bigg(\prod_{e\in E(T)}h_e\bigg)\prod_{e\in R(T)}(1+h_e),
\end{equation*}
where $R(T)$ is a certain set of edges depending on $T$, disjoint from $E(T)$. The important point is that the sum is now over trees, not over arbitrary connected graphs.
Since by definition $\int_{\T^d} h=0$, we note that the bare tree contribution vanishes,
\[\int_{(\T^d)^{S}}\bigg(\prod_{e\in E(T)}h_e\bigg)dx_S=0,\]
and we are thus led to
\begin{equation}\label{eq:tree-phi(S)}
\varphi(S)=\sum_{\substack{T\text{ tree}\\ V(T)=S}}\int_{(\T^d)^S}\bigg(\prod_{e\in E(T)}h_e\bigg)\bigg(\prod_{e\in R(T)}(1+h_e)-1\bigg)dx_S.
\end{equation}
Ordering the edges, we can expand the last factor as
\begin{equation}\label{eq:expand1-h}
\prod_{e\in R(T)}(1+h_e)-1=\sum_{e'\in R(T)}h_{e'}\prod_{\substack{e\in R(T)\\e<e'}}(1+h_e).
\end{equation}
Let us first estimate the product over $R(T)$. Recalling that $1+h=\frac{1}{1+c_0}e^{-\frac\beta NW}$ and that $c_0\ge0$, cf.~\eqref{eq:est-c0}, we can bound
\[\prod_{\substack{e\in R(T)\\e<e'}}(1+h_e)\le \exp\bigg(-\frac\beta N\sum_{\substack{(i,j)\in R(T)\\(i,j)<e'}}W(x_i-x_j)\bigg)\le e^{\frac\beta N|R(T)|\|W_-\|_{L^\infty}},\]
hence, using the crude bounds $|R(T)|\le\binom{|S|}{2}$ and $|S|\le N$,
\[\prod_{\substack{e\in R(T)\\e<e'}}(1+h_e)\le e^{\beta|S|\|W_-\|_{L^\infty}}.\]
Combined with~\eqref{eq:tree-phi(S)} and~\eqref{eq:expand1-h}, this yields
\begin{gather*}
|\varphi(S)|\le e^{\beta|S|\|W_-\|_{L^\infty}}\sum_{\substack{T\text{ tree}\\ V(T)=S}}\sum_{e'\in R(T)}I(T,e'),
\qquad I(T,e'):=\int_{(\T^d)^S}|h_{e'}|\bigg(\prod_{e\in E(T)}|h_e|\bigg)dx_S.
\end{gather*}
For a fixed tree $T$ on $S$ and for $e'\in R(T)$, recalling that $R(T)$ is disjoint from $E(T)$, we note that the graph $T\cup\{e'\}$ contains exactly one cycle $C$. Denote by $\ell\ge3$ its length. The contribution of this cycle in the integral $I(T,e')$ is bounded by
\[\int_{(\T^d)^\ell}\bigg(\prod_{r=1}^\ell |h(x_r-x_{r+1})|\bigg)dx_1\ldots dx_\ell=\Tr((|h|\ast)^\ell)\le
\|h\|_{L^2(\T^d)}^2\|h\|_{L^1(\T^d)}^{\ell-2},\]
with $x_{\ell+1}:=x_1$. Next, all remaining branches in $T\setminus C$ can be integrated from their leaves inward, and we obtain a factor~$\|h\|_{L^1(\T^d)}$ every time we integrate a new vertex. This leads us to
\[I(T,e')\le\|h\|_{L^2(\T^d)}^2\|h\|_{L^1(\T^d)}^{|S|-2}.\]
By Cayley's formula, the number of trees on the set $S$ is $|S|^{|S|-2}$. Using again $|R(T)|\le\binom{|S|}2$, we then obtain
\[|\varphi(S)|\le  e^{\beta|S|\|W_-\|_{L^\infty}}C^{|S|}|S|^{|S|-2}\|h\|_{L^2(\T^d)}^2\|h\|_{L^1(\T^d)}^{|S|-2}.\]
Summing this bound over all subsets $S\subset\{1,\dots,N\}$ with $|S|=k$, and using $\binom{N}{k}\le \frac{N^k}{k!}$, we deduce
\begin{eqnarray}
\sum_{\substack{S\subset\{1,\dots,N\}\\ |S|=k}}|\varphi(S)|
&\le&\binom{N}{k}e^{\beta k\|W_-\|_{L^\infty}}C^{k}k^{k-2}\|h\|_{L^2(\T^d)}^2\|h\|_{L^1(\T^d)}^{k-2}\nonumber\\[-4mm]
&\le&e^{\beta k\|W_-\|_{L^\infty}}C^k N^2\|h\|_{L^2(\T^d)}^2(N\|h\|_{L^1(\T^d)})^{k-2}.\label{eq:main-varphi}
\end{eqnarray}
Combining~\eqref{eq:mayer-Z} and~\eqref{eq:main-varphi}, and noting that
\begin{equation*}
\|h\|_{L^p(\T^d)}=\|e^{-\frac\beta NW}-1\|_{L^p(\T^d)}\le
\frac{\beta}N\|W\|_{L^p(\T^d)}e^{\frac\beta N\|W_-\|_{L^\infty}},
\end{equation*}
we can deduce, provided that $\beta(\|W\|_{L^1(\T^d)}+\|W_-\|_{L^\infty(\T^d)})\ll1$ is small enough,
\[\log Z_h\le \beta^2\|W\|_{L^2(\T^d)}^2\sum_{k=2}^NC^k(\beta\|W\|_{L^1(\T^d)})^{k-2}\le C\beta^2\|W\|_{L^2(\T^d)}^2.\]
This concludes the proof of Theorem~\ref{th:partition}(i).\qed

\subsection{Partition function estimate beyond $L^2$ interactions}\label{sec:noL2case}

This section is devoted to the proof of Theorem~\ref{th:partition}(ii)-(iii) in the case $W\notin L^2(\T^d)$. We argue by an interpolation argument based on the~$L^2$ case~(i) and we split the proof into two steps.
Denote by $Z_{N,\beta}(U)$ the partition function $\bar Z_{N,\beta}$ with potential $W$ replaced by an arbitrary even function~$U$.

\medskip\noindent
{\bf Step~1:} Partition function estimate without cancellations: for any even periodic function $U$,
\begin{equation}\label{eq:bound-naive}
Z_{N,\beta}(U)\le \exp\Big(\beta N\|U_-\|_{L^1(\T^d)}e^{\beta\|U_-\|_{L^\infty}}\Big).
\end{equation}
For any $1\le n\le N$, starting from the decomposition
\begin{multline*}
\int_{(\T^d)^n}
\exp\Bigl(-\frac\beta{2N}\sum_{i\ne j}^nU(x_i-x_j)\Bigr)\,dx_1\ldots dx_n\\
=\int_{(\T^d)^{n-1}}\exp\Big(-\frac\beta{2N}\sum_{i\ne j}^{n-1}U(x_i-x_j)\Big)\bigg(\int_{\T^d}\exp\Big(-\frac\beta N\sum_{i=1}^{n-1}U(x_i-y)\Big)dy\bigg)dx_1\ldots dx_{n-1},
\end{multline*}
we can bound by convexity
\begin{eqnarray*}
\lefteqn{\int_{(\T^d)^n}
\exp\Bigl(-\frac\beta{2N}\sum_{i\ne j}^nU(x_i-x_j)\Bigr)\,dx_1\ldots dx_n}\\
&\le& \int_{(\T^d)^{n-1}}\exp\Big(-\frac\beta{2N}\sum_{i\ne j}^{n-1}U(x_i-x_j)\Big)\bigg(\frac1{n-1}\sum_{i=1}^{n-1}\int_{\T^d}e^{-\beta\frac{n-1}NU(x_i-y)}dy\bigg)dx_1\ldots dx_{n-1}\\
&=&\Big(\int_{\T^d}e^{-\beta\frac{n-1}NU}\Big)\int_{(\T^d)^{n-1}}\exp\Big(-\frac\beta{2N}\sum_{i\ne j}^{n-1}U(x_i-x_j)\Big)dx_1\ldots dx_{n-1},
\end{eqnarray*}
and thus, by a direct iteration,
\begin{equation*}
Z_{N,\beta}(U)
\le\prod_{n=1}^{N-1}\int_{\T^d}e^{\beta\frac{n}NU_-}.
\end{equation*}
Using the elementary inequality $e^t\le 1+te^t$ for $t\ge0$, this yields the claim.

\medskip\noindent
{\bf Step~2:} Conclusion.\\
Given $M>0$, we decompose the interaction potential as
\[W=U_M+V_M,\qquad U_M:=W\mathds1_{|W|\le M}-\int_{\T^d}W\mathds1_{|W|\le M}.\]
Since $\int_{\T^d}W=0$, this also gives
\[V_M=W\mathds1_{|W|> M}-\int_{\T^d}W\mathds1_{|W|> M}.\]
By the Cauchy-Schwarz inequality,
\[\bar Z_{N,\beta}\le Z_{N,2\beta}(U_M)^\frac12Z_{N,2\beta}(V_M)^\frac12.\]
Using Theorem~\ref{th:partition}(i) for $Z_{N,2\beta}(U_M)$, and the crude bound~\eqref{eq:bound-naive} for $Z_{N,2\beta}(V_M)$, we obtain
\begin{equation}\label{eq:bound-Z-interpol}
\bar Z_{N,\beta}\le \exp\Big(C\beta^2\|U_M\|_{L^2(\T^d)}^2+C\beta N\|V_M\|_{L^1(\T^d)}e^{\beta\|(V_M)_-\|_{L^\infty}}\Big),
\end{equation}
provided that $\beta(\|U_M\|_{L^1}+\|(U_M)_-\|_{L^\infty})\ll1$ is sufficiently small.

Depending on the properties of $W$, estimating the norms of $U_M$ and $V_M$ and optimizing the choice of~$M$, the conclusion follows easily from~\eqref{eq:bound-Z-interpol}.
Let us focus on the critical case $W\in L^{2,\infty}(\T^d)$. By the standard layer-cake estimates for weak-$L^2$ functions, we find
\[\|U_M\|_{L^2(\T^d)}\le2\Big(\int_{\T^d}|W|^2\mathds1_{|W|\le M}\Big)^\frac12\lesssim\|W\|_{L^{2,\infty}(\T^d)}\log^\frac12\Big(2+\frac{M}{\|W\|_{L^{2,\infty}(\T^d)}}\Big),\]
and
\[\|V_M\|_{L^1(\T^d)}\le2\int_{\T^d}|W|\mathds1_{|W|>M}\lesssim M^{-1}\|W\|_{L^{2,\infty}(\T^d)}^2.\]
Moreover,
\begin{eqnarray*}
\|U_M\|_{L^1(\T^d)}+\|V_M\|_{L^1(\T^d)}&\lesssim&\|W\|_{L^1(\T^d)},\\
\|(U_M)_-\|_{L^\infty(\T^d)}+\|(V_M)_-\|_{L^\infty(\T^d)}&\lesssim&\|W\|_{L^1(\T^d)}+\|W_-\|_{L^\infty(\T^d)}.
\end{eqnarray*}
Inserting these estimates into~\eqref{eq:bound-Z-interpol} and choosing $M=N\|W\|_{L^{2,\infty}(\T^d)}$, we obtain
\begin{equation*}
\bar Z_{N,\beta}\le \exp\Big(C\beta^2\|W\|_{L^{2,\infty}(\T^d)}^2\log N+C\Big)\le C N^{C\beta^2\|W\|_{L^{2,\infty}(\T^d)}^2},
\end{equation*}
provided that $\beta(\|W\|_{L^1}+\|W_-\|_{L^\infty})\ll1$ is small enough. This concludes the proof.
\qed

\subsection{Proof of Proposition~\ref{cor:equ-fl}}\label{sec:divergence}
We start with the proof of~(i). For $W\in L^2(\T^d)$ as in the statement, given $M>0$, let us decompose this time the interaction potential as
\[W=W_M+R_M,\]
for some $W_M\in C(\T^d)$ with $\int_{\T^d}W_M=0$, $\|(W_M)_-\|_{L^\infty}\lesssim\|W_-\|_{L^\infty}$, and $W_M\to W$ in~$L^2(\T^d)$ as~$M\uparrow\infty$.
Let then $\bar Z_{N,\beta}^M:=Z_{N,\beta}(W_M)$ stand for the partition function with $W$ replaced by $W_M$. Using the elementary bound $|e^{-t}-e^{-s}|\le |t-s|(e^{-t}+e^{-s})$, the Cauchy-Schwarz inequality yields
\begin{equation*}
|\bar Z_{N,\beta}-\bar Z_{N,\beta}^M|
\le
\Big((\bar Z_{N,2\beta})^\frac12+(\bar Z_{N,2\beta}^M)^\frac12\Big)
\bigg(\int_{(\T^d)^N}\Big|\frac{\beta}{2N}\sum_{i\neq j}^NR_M(x_i-x_j)\Big|^2dx_1\ldots dx_N\bigg)^\frac12.
\end{equation*}
Expanding the square and using $\int_{\T^d}R_M=0$, this implies
\begin{equation*}
|\bar Z_{N,\beta}-\bar Z_{N,\beta}^M|
\le
\beta\Big((\bar Z_{N,2\beta})^\frac12+(\bar Z_{N,2\beta}^M)^\frac12\Big)
\Big(\int_{\T^d}|R_M|^2\Big)^\frac12.
\end{equation*}
For $\beta$ small enough, applying Theorem~\ref{th:partition}(i) and recalling $R_M=W-W_M\to0$ in $L^2(\T^d)$, we deduce
\begin{equation}\label{eq:ZMapprox}
\lim_{M\uparrow\infty}\sup_N|\bar Z_{N,\beta}-\bar Z_{N,\beta}^M|=0.
\end{equation}
For fixed $M$, as $W_M$ is continuous and bounded with $\int_{\T^d}W_M=0$, we can now appeal to~\cite[Th\'eor\`eme~(C)]{Benarous-Brunaud}: provided $\beta\|(W_M)_-\|_{L^\infty}<1$, we obtain
\begin{eqnarray*}
\lim_{N\uparrow\infty}\bar Z_{N,\beta}^M
&=&e^{\frac\beta{2}W_M(0)}\lim_{N\uparrow\infty}\int_{(\T^d)^N}\exp\Big(-\frac\beta{2N}\sum_{i,j=1}^NW_M(x_i-x_j)\Big)dx_1\ldots dx_N\\
&=&e^{\frac\beta{2}W_M(0)}\det{}_{\! L^2(\T^d)}\big(\Id+\tfrac\beta2 W_M\ast\big)^{-\frac12},
\end{eqnarray*}
that is, computing the determinant and expanding $W_M(0)=\sum_\xi\widehat W_M(\xi)$,
\begin{equation}\label{eq:limZM}
\lim_{N\uparrow\infty}\bar Z_{N,\beta}^M=\exp\bigg(\sum_{\xi\ne0}\Big(\tfrac\beta2\widehat W_M(\xi)-\log(1+\tfrac\beta2\widehat W_M(\xi))\Big)\bigg).
\end{equation}
Combined with~\eqref{eq:ZMapprox}, this yields the claim.

We turn to the proof of~(ii).
The above regularized potential $W_M$ can be constructed for instance as $W_M=\chi_M\ast\chi_M\ast W$, with $\chi_M(x)=M^d\chi(Mx)$, for some even cut-off function~$\chi\in C^\infty(\T^d)$ with~$\int_{\T^d}\chi=1$.
It is then a consequence of Jensen's inequality that
\[\bar Z_{N,\beta}\ge\int_{(\T^d)^N}\exp\Big(-\frac\beta{2N}\sum_{i\ne j}^N(\chi_M\ast\chi_M\ast W)(x_i-x_j)\Big)dx_1\ldots dx_N=\bar Z_{N,\beta}^M.\]
%and thus in particular, for all $M$, recalling the choice $\widehat W_M(\xi)=\widehat W(\xi)\hat\chi(\frac\xi M)^2$,
%\[\bar Z_{N,\beta}\ge \bar Z_{N,\beta}^M.\]
Using the lower bound $t-\log(1+t)\ge \frac14t^2$ for $0<t\le1$, we easily check that the right-hand side in~\eqref{eq:limZM} diverges as $M\uparrow\infty$ if $W\notin L^2(\T^d)$. This allows to conclude $\lim_N\bar Z_{N,\beta}=\infty$.\qed

\section{A priori correlation estimates}
This section is devoted to a priori estimates on particle correlations. As e.g.\@ in~\cite{BGS-17,DSR-21}, we define correlation functions as the Hoeffding coefficients measuring the deviation of the Liouville solution~$F_N$ from the mean-field equilibrium~$M_\beta^{\otimes N}$. More precisely, we let $\pi_\beta$ denote the orthogonal projection onto the span of the mean-field equilibrium $M_\beta$ in $L^2(1/M_\beta)$, that is,
\[\pi_\beta h=M_\beta\int_{\Dd}h,\]
and we define for all $1\le m\le N$,
\begin{equation}\label{eq:HNm-def}
H_{N,m}=(\Id-\pi_\beta)^{\otimes m}F_{N,m},
\end{equation}
where we recall that $F_{N,m}$ is the $m$-particle marginal~\eqref{eq:marginalFNm}. More explicitly, this means
\begin{equation}\label{eq:HNm-marg}
H_{N,m}(z_1,\ldots,z_m)=\sum_{j=0}^{m}(-1)^{m-j}\sum_{\sigma\in P^m_j}F_{N,j}(z_\sigma)M_\beta^{\otimes m-j}(z_{[m]\setminus\sigma}),
\end{equation}
where $P_j^m$ stands for the set of all $j$-element subsets of $[m]:=\{1,\ldots,m\}$ and where we write $z_J:=(z_{i_1},\ldots,z_{i_l})$ for an index set $J=\{i_1,\ldots,i_l\}$.
Using the orthogonality properties of these correlation functions, together with our estimates on partition functions, we establish the following uniform-in-time a priori estimates.

\begin{prop}\label{prop:apriori-lincorrel}
Let $K=-\nabla W\in L^1(\T^d)^d$ for some even potential $W\in L^2(\T^d)$ with $W_-\in L^\infty(\T^d)$ and $\int_{\T^d}W=0$. Let $F_N$ be a global bounded weak solution of~\eqref{eq:Liouville} with initial data~\eqref{eq:initial} such that~\eqref{eq:conserv} holds. Provided that~$\beta\|W\|_{L^1(\T^d)}\ll1$ is sufficiently small, we have for all $1\le m\le N$ and $t\ge0$,
\begin{equation}\label{eq:est-HNm-aprioriL2}
\Big(\int_{\Dd^m}\frac{|H_{N,m}(t)|^{2}}{M_\beta^{\otimes m}}\Big)^\frac1{2}\lesssim_m N^{-\frac{m-1}2}.
\end{equation}
More generally, given $2\le p<\infty$ and $\e>0$, provided that $\beta\|W\|_{L^1(\T^d)}\ll_{p,\e}1$ is sufficiently small (depending on $p,\e$), we further have for all $1\le m\le N$ and $t\ge0$,
\begin{equation}\label{eq:est-HNm-aprioriLp}
\Big(\int_{\Dd^m}\frac{|H_{N,m}(t)|^{p}}{|M_\beta^{\otimes m}|^{p-1}}\Big)^\frac1{p}\lesssim_{m,p,\e} N^{\frac{m}{2}(1-\frac2p)(1+\e)-\frac{m-1}2}.
\end{equation}
The smallness conditions on $\beta$ can be dropped if $W$ is the repulsive logarithmic kernel. If the assumption $W\in L^2(\T^d)$ is replaced by $W\in L^{2,\infty}(\T^d)$, then the same estimates~\eqref{eq:est-HNm-aprioriL2} and~\eqref{eq:est-HNm-aprioriLp} hold with an additional factor $N^{C\beta^2}$.
\end{prop}

The proof is postponed to Section~\ref{sec:prop:apriori-lincorrel}; we start with some more basic a priori estimate on the Liouville solution.

\subsection{A priori estimate on Liouville solution} 
The following uniform-in-time a priori estimate shows that the Liouville solution $F_N$ remains over time a fluctuation around the mean-field equilibrium~$M_\beta^{\otimes N}$. This property follows from our estimates on partition functions, which indeed yield a control on the density of the interacting Gibbs equilibrium relative to the product mean-field measure; see~\eqref{eq:moment-MN}.

\begin{lem}\label{lem:apriori-Liouv}
Let $K=-\nabla W\in L^1(\T^d)^d$ for some even potential $W\in L^2(\T^d)$ with $W_-\in L^\infty(\T^d)$ and $\int_{\T^d}W=0$. Let $F_N$ be a bounded weak solution of~\eqref{eq:Liouville} with initial data~\eqref{eq:initial} such that~\eqref{eq:conserv} holds. Given~$2\le p<\infty$, provided that $\beta\|W\|_{L^1(\T^d)}\ll_p1$ is sufficiently small (depending on $p$), we have for all $t\ge0$,
\begin{equation}\label{eq:estim-FNeq}
\int_{\Dd^N}\frac{|F_N(t)|^{p}}{|M_\beta^{\otimes N}|^{p-1}}\,\lesssim_p\,N^{\frac p2}.
\end{equation}
The smallness condition on $\beta$ can be dropped if $W$ is the repulsive logarithmic kernel. If the assumption $W\in L^2(\T^d)$ is replaced by $W\in L^{2,\infty}(\T^d)$, then the same estimate holds with an additional factor~$N^{C\beta^2}$.
\end{lem}

\begin{proof}
Let $2\le p<\infty$. By the Cauchy-Schwarz inequality and the conservation estimate~\eqref{eq:conserv}, we find for all~$t\ge0$,
\begin{equation*}
\int_{\Dd^N}\frac{|F_N(t)|^{p}}{|M_\beta^{\otimes N}|^{p-1}}
\le\bigg(\int_{\Dd^N}\frac{|F_N^\circ|^{2p}}{|M_{N,\beta}|^{2p-1}}\bigg)^\frac12\bigg(\int_{\Dd^N}\frac{|M_{N,\beta}|^{2p-1}}{|M_\beta^{\otimes N}|^{2p-2}}\bigg)^\frac1{2}.
\end{equation*}
By the choice of initial data, the first factor is bounded by
\begin{eqnarray*}
\bigg(\int_{\Dd^N}\frac{|F_N^\circ|^{2p}}{|M_{N,\beta}|^{2p-1}}\bigg)^\frac12
&=&\bigg(\int_{\Dd^N}\Big|\sum_{i=1}^Nf_\circ(z_i)\Big|^{2p} M_{N,\beta}\bigg)^\frac12\\
&\le&\bigg(\int_{\Dd^N}\Big|\sum_{i=1}^Nf_\circ(z_i)\Big|^{4p}M_\beta^{\otimes N}\bigg)^\frac14\bigg(\int_{\Dd^N}\frac{|M_{N,\beta}|^2}{M_\beta^{\otimes N}}\bigg)^\frac14,
\end{eqnarray*}
and thus, expanding the power and using $\int_\Dd f_\circ M_\beta=0$,
\begin{equation}\label{eq:init-moments}
\bigg(\int_{\Dd^N}\frac{|F_N^\circ|^{2p}}{|M_{N,\beta}|^{2p-1}}\bigg)^\frac12
\lesssim N^\frac p2\bigg(\int_{\Dd}|f_\circ|^{4p}M_\beta\bigg)^\frac14\bigg(\int_{\Dd^N}\frac{|M_{N,\beta}|^2}{M_\beta^{\otimes N}}\bigg)^\frac14.
\end{equation}
Inserting this into the above, and decomposing $M_{N,\beta}=M_\beta^{\otimes N}\bar M_{N,\beta}$ in terms of the spatial Gibbs correction
\begin{equation}\label{eq:Gibbs-spatial}
\bar M_{N,\beta}(x_1,\ldots,x_N):=\bar Z_{N,\beta}^{-1}\exp\Big(-\frac{\beta}{2N}\sum_{i\neq j}^NW(x_i-x_j)\Big)~\in\Pc((\T^d)^N),
\end{equation}
we are led to
\begin{equation*}
\int_{\Dd^N}\frac{|F_N(t)|^{p}}{|M_\beta^{\otimes N}|^{p-1}}
\lesssim N^\frac p2\bigg(\int_{\Dd}|f_\circ|^{4p}M_\beta\bigg)^\frac14\bigg(\int_{(\T^d)^N}|\bar M_{N,\beta}|^{2}\bigg)^{\frac14}\bigg(\int_{(\T^d)^N}|\bar M_{N,\beta}|^{2p-1}\bigg)^{\frac12}.
\end{equation*}
It remains to estimate the last two factors. They amount to ratios of partition functions,
\[\int_{\Dd^N}|\bar M_{N,\beta}|^{q}=\frac{\bar Z_{N,q\beta}}{(\bar Z_{N,\beta})^{q}},\]
and the conclusion then follows from Theorem~\ref{th:partition} --- or from~\cite[Theorem~2.5]{2506.22083} for the repulsive logarithmic case.
\end{proof}

\subsection{Proof of Proposition~\ref{prop:apriori-lincorrel}}\label{sec:prop:apriori-lincorrel}
By definition of correlation functions, cf.~\eqref{eq:HNm-def}, we get by orthogonality
\[\sum_{m=1}^N\binom{N}{m}\int_{\Dd^m}\frac{|H_{N,m}|^2}{M_\beta^{\otimes m}}\,=\,\int_{\Dd^N}\frac{|F_N|^2}{M_\beta^{\otimes N}}.\]
As the right-hand side is $O(N)$ by Lemma~\ref{lem:apriori-Liouv} provided that $\beta\|W\|_{L^1(\T^d)}\ll1$ is sufficiently small, this proves the desired $L^2$ estimate in the statement.
We now turn to higher norms.
For any $p\ge1$, the definition of $H_{N,m}$ in terms of marginals, cf.~\eqref{eq:HNm-marg}, trivially yields
\[\int_{\Dd^m}\frac{|H_{N,m}|^{p}}{|M_\beta^{\otimes m}|^{p-1}}
\lesssim_m\sum_{j=0}^m\int_{\Dd^j}\frac{|F_{N,j}|^{p}}{|M_\beta^{\otimes j}|^{p-1}}
\lesssim_m \int_{\Dd^N}\frac{|F_{N}|^{p}}{|M_\beta^{\otimes N}|^{p-1}},\]
where the last estimate follows from H\"older's inequality.
The right-hand side is $O(N^{\frac p2})$ by Lemma~\ref{lem:apriori-Liouv} provided that $\beta\|W\|_{L^1(\T^d)}\ll_p1$ is sufficiently small (depending on $p$).
Interpolating with the $L^2$ bound yields the conclusion.\qed

\section{Linearized mean field}
We turn to the proof of Theorem~\ref{th:main}.
By the conservation estimate~\eqref{eq:conserv} for $q=1$, we find that~$F_{N,1}$ is uniformly bounded in $L^\infty(\R^+;L^1(\Dd))$ as $N\uparrow\infty$.
By weak compactness, up to a subsequence, we thus find
\[F_{N,1}\to f\quad\text{vaguely in $L^\infty(\R^+;\mathcal M(\Dd))$,}\]
for some limit $f$.
(In case $W\in L^2(\T^d)$, by Proposition~\ref{prop:apriori-lincorrel} with $m=1$, this can be turned into \mbox{weak-*} convergence in $L^\infty(\R^+;L^2(1/M_\beta))$.)
In order to deduce a corresponding convergence for higher marginals, we recall the a priori correlation estimates of Proposition~\ref{prop:apriori-lincorrel}: for $W\in L^{2,\infty}(\T^d)$ and $\beta\|W\|_{L^1(\T^d)}\ll1$ sufficiently small, we find for all $m\ge2$,
\[H_{N,m}\to0\quad\text{strongly in $L^\infty(\R^+;L^2(1/M_\beta^{\otimes m}))$}.\]
Combining this with a cluster expansion, inverting~\eqref{eq:HNm-marg}, we thus get for marginals, along the subsequence,
\begin{multline*}
F_{N,m}=\sum_{j=1}^m\sum_{\sigma\in P_j^{m}}H_{N,j}(z_\sigma)M_\beta^{\otimes m-j}(z_{[m]\setminus\sigma})\\[-4mm]
\to\sum_{j=1}^mf(z_j) M_\beta^{\otimes m-1}(z_{[m]\setminus\{j\}})\quad\text{vaguely in $L^\infty(\R^+;\mathcal M(\Dd^m))$}.
\end{multline*}
In order to prove Theorem~\ref{th:main}, it thus remains to characterize the extracted limit $f$.
For this purpose, we note that the first marginal $F_{N,1}$ satisfies the following exact equation,
\begin{equation*}
\partial_t F_{N,1}+v\cdot\nabla_xF_{N,1}
+\frac{N-1}N\nabla_v M_\beta\cdot (K\ast F_{N,1})
=-\frac{N-1}N\nabla_v\cdot\int_\Dd K(x-x_*) H_{N,2}(z,z_*)\,dz_*.
\end{equation*}
We proceed by passing to the limit in the weak formulation. If we could show that the right-hand side converges to $0$ in the distributional sense, then we could conclude that the limit $f$ is a weak solution of the expected linearized Vlasov equation~\eqref{eq:Vlasov}, and uniqueness would yield the conclusion. It is thus sufficient to prove, for all $t\ge0$,
\begin{equation}\label{eq:todo-vanishHN2}
\int_{\Dd}\Big|\int_\Dd K(x_1-x_2) H_{N,2}(t,z_1,z_2)\,dz_2\Big|\,dz_1\to0.
\end{equation}
%We split the proof into two steps, separately considering the subcritical case $W\in L^2(\T^d)$ and the critical case $W\in L^{2,\infty}(\T^d)$.
%
%\medskip\noindent
%{\bf Step~1:} Subcritical case.\\
%Let $W\in L^2(\T^d)$.
Assume that $W\in L^{2,\infty}(\T^d)$ and $K\in L^p(\T^d)$ for some $p>\frac43$.
Given $\e>0$, provided that $\beta\|W\|_{L^1(\T^d)}\ll_{p,\e}1$ is sufficiently small (depending on $p,\e$), the bound~\eqref{eq:est-HNm-aprioriLp} of Proposition~\ref{prop:apriori-lincorrel} yields
\begin{equation}\label{eq:bounHN2-1}
\int_{\Dd}\Big|\int_\Dd K(x_1-x_2) H_{N,2}(t,z_1,z_2)\,dz_2\Big|\,dz_1
\lesssim \|K\|_{L^{p}(\T^d)}N^{C\beta^2+(1-\frac2{p'})(1+\e)-\frac12}.
\end{equation}
As $p'<4$, note that $1-\frac2{p'}<\frac12$. Choosing $\e$ small enough, and then $\beta$ sufficiently small so that the power of $N$ is negative, the conclusion~\eqref{eq:todo-vanishHN2} follows.\qed

\subsection*{Acknowledgements}
MD acknowledges financial support from the European Union (ERC, PASTIS, Grant Agreement n$^\circ$101075879).\footnote{{Views and opinions expressed are however those of the authors only and do not necessarily reflect those of the European Union or the European Research Council Executive Agency. Neither the European Union nor the granting authority can be held responsible for them.}} PEJ was partially supported by NSF DMS Grant 2508570.

\appendix
\section{Whole-space setting with confinement}\label{app:Rd}
For completeness, we briefly explain how our results easily adapt to the corresponding particle dynamics on the whole space $\R^d$ with a confining potential~$V:\R^d\to\R$. Redefining $\Dd=\R^d\times\R^d$, we now consider the following Liouville equation on $\Dd^N$,
\begin{equation}\label{eq:Liouville-re}
\partial_tF_N
+\sum_{j=1}^Nv_j\cdot\nabla_{x_j}F_N
-\sum_{j=1}^N\nabla V(x_j)\cdot\nabla_{v_j}F_N+\frac1N\sum_{j\ne l}^NK(x_j-x_l)\cdot\nabla_{v_j}F_N=0.
\end{equation}
The associated Gibbs equilibrium reads
\[M_{N,\beta}(z_1,\ldots,z_N):=Z_{N,\beta}^{-1}\exp\bigg(-\frac\beta2\sum_{j=1}^N|v_j|^2-\beta\sum_{j=1}^NV(x_j)-\frac{\beta}{2N}\sum_{j\ne l}^NW(x_j-x_l)\bigg),\]
while the mean-field equilibrium $M_\beta\in\Pc(\Dd)$ is now solution of the
fixed-point equation
\begin{equation}\label{eq:mubeta}
M_\beta(x,v)=Z_\beta^{-1} e^{-\frac\beta2|v|^2-\beta V(x)-\beta W\ast M_\beta(x)}.
\end{equation}
With this redefined notation, we consider the Liouville equation~\eqref{eq:Liouville-re} with initial condition~\eqref{eq:initial}, and Theorem~\ref{th:main} is then adapted as follows. Note that the corresponding linearized Vlasov equation~\eqref{eq:Vlasov-re} below further retains the term involving $K\ast M_\beta$, which no longer vanishes in this case.

\begin{theor}\label{th:main-re}
Let the interaction kernel be given by $K=-\nabla W$, for some even $W\in (L^{2,\infty}+L^\infty)(\R^d)$ with bounded negative part $W_-\in L^\infty(\R^d)$, and assume $K\in (L^p+L^\infty)(\R^d)$ for some~$p>\frac43$. Let the confining potential $V\in L^1_\loc(\R^d)$ satisfy~$\int_{\R^d}e^{-V}<\infty$ and $\inf V>-\infty$.
%Further assume that we are in one of the following two situations:
%\begin{enumerate}[---]
%\item \emph{Subcritical case:} $W\in (L^2+L^\infty)(\R^d)$ and $K\in (L^p+L^\infty)(\R^d)$ for some~$p>\frac43$;
%\item \emph{Critical case:} $W\in (L^{2,\infty}+L^\infty)(\R^d)$ is the repulsive Riesz kernel with exponent $\frac d2$, that is, $W(x)=|x|^{-d/2}$, in space dimension $d>2$.
%\end{enumerate}
Then there exists $\beta_0>0$ such that for any $\beta\in(0,\beta_0)$ the mean-field fixed-point equation~\eqref{eq:mubeta} has a unique solution~$M_\beta\in L^\infty\cap\Pc(\R^d)$ and such that the following holds. In the special case of the repulsive logarithmic kernel $W(x)=-\log(|x|)$, one can take $\beta_0=\infty$.

Let $F_N$ be a global bounded weak solution to~\eqref{eq:Liouville-re} with initial data~\eqref{eq:initial} such that~\eqref{eq:conserv} holds. For every $m\ge1$, we have
\[F_{N,m}\to \sum_{j=1}^m f(t,z_j)\prod_{l:l\ne j}^m M_\beta(z_l)\qquad\text{in the sense of distributions on $\R^+\times\Dd^m$},\]
where $f$ is the unique weak solution of the linearized Vlasov equation
\begin{equation}\label{eq:Vlasov-re}
\left\{\begin{array}{l}
\partial_tf+v\cdot\nabla_xf+(K\ast f)\cdot\nabla_vM_\beta+(K\ast M_\beta)\cdot\nabla_vf=0,\\
f|_{t=0}=f_\circ.
\end{array}\right.
\end{equation}
\end{theor}

\subsection{Well-posedness of mean-field fixed-point equation}
For completeness, we first quickly argue that under the assumptions of Theorem~\ref{th:main-re} the fixed-point equation~\eqref{eq:mubeta} defining the mean-field equilibrium $M_\beta$ is indeed well-posed, as claimed in the statement. In case of the repulsive logarithmic interaction, the corresponding well-posedness statement holds for all $\beta<\infty$ and follows instead from the strict convexity of the free energy.

\begin{lem}
Let $W\in (L^1+L^\infty)(\R^d)$ be even with $W_-\in L^\infty(\R^d)$, and let $V\in L^1_\loc(\R^d)$ satisfy~$\int_{\R^d}e^{-V}<\infty$ and $\inf V>-\infty$.
Then, for $\beta>0$ small enough, the mean-field fixed-point equation~\eqref{eq:mubeta} has a unique solution~$M_\beta\in L^\infty\cap\Pc(\R^d)$.
\end{lem}

\begin{proof}
Splitting $M_\beta=\gamma_\beta(v)\bar M_\beta(x)$ in terms of the Maxwellian velocity density $\gamma_\beta(v)=(\frac{\beta}{2\pi})^{\frac d2}e^{-\frac\beta2|v|^2}$, the fixed-point equation~\eqref{eq:mubeta} reads
\begin{equation}\label{eq:rhobeta}
\bar M_\beta=S_\beta(\bar M_\beta):=\frac{e^{-\beta V-\beta W\ast\bar M_\beta}}{\int_{\R^d}e^{-\beta V-\beta W\ast\bar M_\beta}}.
\end{equation}
We split the proof into two steps.

\medskip\noindent
{\bf Step~1:} Proof that
the reference measure
$\eta_\beta :=\frac{e^{-\beta V}}{\int_{\R^d}e^{-\beta V}}$
satisfies
\begin{equation}\label{eq:est-etabeta}
\sup_{0<\beta\le1}\|\eta_\beta\|_{L^\infty(\R^d)}<\infty.
\end{equation}
Given $1\le q<\infty$, starting from $\|\eta_\beta\|_{L^q(\R^d)}^q=\frac{\int_{\R^d}e^{-q\beta V}}{(\int_{\R^d}e^{-\beta V})^q}$,
we find
\begin{eqnarray*}
\partial_\beta\log\|\eta_\beta\|_{L^q(\R^d)}^q
&=&\frac{q\int_{\R^d}Ve^{-\beta V}}{\int_{\R^d}e^{-\beta V}}
-\frac{q\int_{\R^d}Ve^{-p\beta V}}{\int_{\R^d}e^{-p\beta V}}\\
&=&-q\int_{\beta}^{q\beta}\partial_\lambda\Big( \frac{\int_{\R^d}Ve^{-\lambda V}}{\int_{\R^d}e^{-\lambda V}}\Big)d\lambda\\
&=&q\int_{\beta}^{q\beta}\bigg( \frac{\int_{\R^d}V^2e^{-\lambda V}}{\int_{\R^d}e^{-\lambda V}}-\Big(\frac{\int_{\R^d}Ve^{-\lambda V}}{\int_{\R^d}e^{-\lambda V}}\Big)^2\bigg)d\lambda,
\end{eqnarray*}
which is nonnegative by Jensen's inequality. By the resulting monotonicity, we obtain
\[\sup_{0<\beta\le1}\|\eta_\beta\|_{L^q(\R^d)}\le\|\eta_1\|_{L^q(\R^d)}.\]
Letting $q\uparrow\infty$ and recalling the assumptions on $V$, the claim~\eqref{eq:est-etabeta} follows.

\medskip\noindent
{\bf Step~2:} Contraction estimate.\\
Given $\rho,\rho'\in L^1\cap\Pc(\R^d)$, we can decompose
\[S_\beta(\rho)-S_\beta(\rho')
=\frac{e^{-\beta V}(e^{-\beta W\ast\rho}-e^{-\beta W\ast\rho'})}{\int_{\R^d}e^{-\beta V-\beta W\ast\rho}}
-e^{-\beta V-\beta W\ast\rho'}\frac{\int_{\R^d}e^{-\beta V}(e^{-\beta W\ast\rho}-e^{-\beta W\ast\rho'})}{(\int_{\R^d}e^{-\beta V-\beta W\ast\rho})(\int_{\R^d}e^{-\beta V-\beta W\ast\rho'})}.\]
Using the elementary bound $|e^{-t}-e^{-s}|\le e^B|t-s|$ for $t,s\ge-B$, and recalling $W_-\in L^\infty(\R^d)$, we deduce
\[|S_\beta(\rho)-S_\beta(\rho')|
\le \beta e^{C\beta}e^{-\beta V}\bigg(\frac{|W\ast(\rho-\rho')|}{\int_{\R^d}e^{-\beta V-\beta W\ast\rho}}
+\frac{\int_{\R^d}e^{-\beta V}|W\ast(\rho-\rho')|}{(\int_{\R^d}e^{-\beta V-\beta W\ast\rho})(\int_{\R^d}e^{-\beta V-\beta W\ast\rho'})}\bigg),\]
where henceforth $C$ stands for a generic constant only depending on $V,W$ (not on $\beta$).
Now note that Jensen's inequality yields
\begin{equation*}
\int_{\R^d}e^{-\beta V-\beta W\ast\rho}
\ge\exp\Big(-\beta\int_{\R^d}|W\ast\rho|\,\eta_\beta\Big)\int_{\R^d}e^{-\beta V},
\end{equation*}
and thus, recalling the estimate of Step~1 and the assumption $W\in( L^1+L^\infty)(\R^d)$, for $0<\beta\le1$,
\begin{equation*}
\int_{\R^d}e^{-\beta V-\beta W\ast\rho}
\ge e^{-C\beta}\int_{\R^d}e^{-\beta V}.
\end{equation*}
Inserting this into the above yields
\[|S_\beta(\rho)-S_\beta(\rho')|
\le \beta e^{C\beta}e^{-\beta V}\bigg(\frac{|W\ast(\rho-\rho')|}{\int_{\R^d}e^{-\beta V}}
+\frac{\int_{\R^d}e^{-\beta V}|W\ast(\rho-\rho')|}{(\int_{\R^d}e^{-\beta V})^2}\bigg),\]
and thus
\[\|S_\beta(\rho)-S_\beta(\rho')\|_{L^1}
\le 2\beta e^{C\beta}\int_{\R^d}|W\ast(\rho-\rho')|\,\eta_\beta.\]
Using again the estimate of Step~1 and the assumption $W\in(L^1+L^\infty)(\R^d)$, we are led to
\[\|S_\beta(\rho)-S_\beta(\rho')\|_{L^1}
\le C\beta e^{C\beta}\|\rho-\rho'\|_{L^1}.\]
For $\beta<(eC)^{-1}$, the map $S_\beta$ is thus a contraction on $L^1\cap \Pc(\R^d)$. This ensures the existence of a unique solution $\bar M_\beta\in L^1\cap\Pc(\R^d)$ of~\eqref{eq:rhobeta}, and its boundedness easily follows.
\end{proof}

\subsection{Proof of Theorem~\ref{th:main-re}}
The proof of Theorem~\ref{th:main} can be repeated in the present setting with confinement once the a priori estimates of Lemma~\ref{lem:apriori-Liouv} on the Liouville solution are suitably adapted. For this purpose, the key input is again the partition function estimate established in Section~\ref{sec:partition}, which remains valid in the following form.

\begin{theor}[High-temperature bound]
Let $W\in (L^1+L^\infty)(\R^d)$ be even with $W_-\in L^\infty(\R^d)$, let $V\in L^1_\loc(\R^d)$ satisfy $\int_{\R^d} e^{-V}<\infty$ and $\inf V>-\infty$, and assume that $\beta\|W\|_{(L^1+L^\infty)(\T^d)}\ll1$ is sufficiently small. Denote by $\bar M_\beta:=\int_{\R^d}M_\beta(\cdot,v)dv\in L^\infty\cap\Pc(\R^d)$ the spatial equilibrium measure, define the~$\bar M_\beta$-centered interaction kernel
\[W_\beta(x,y)=W(x-y)-W\ast\bar M_\beta(x)-W\ast\bar M_\beta(y)+\iint_{\R^d\times\R^d}W(x-y)\bar M_\beta(x)\bar M_\beta(y)\,dxdy,\]
and the associated partition function
\[\widehat Z_{N,\beta}:=\int_{(\R^d)^N}\exp\Big(-\frac{\beta}{2N}\sum_{j\ne l}^NW_\beta(x_j,x_l)\Big)\,\bar M_\beta^{\otimes N}.\]
Then we have for all $N\ge1$,
\[1\le \widehat Z_{N,\beta}\le \left\{\begin{array}{ll}
\exp(Ce^{C\beta}),&\text{if $W\in(L^{2}+L^\infty)(\R^d)$},\\
C N^{C\beta^2},&\text{if $W\in(L^{2,\infty}+L^\infty)(\R^d)$}.
\end{array}\right.\]
\end{theor}

These estimates are obtained by repeating the argument of Section~\ref{sec:noL2case}, simply using that the centered kernel~$W_\beta$ satisfies the corresponding cancellation property,
\[\int_{\R^d} W_\beta(x,y)\bar M_\beta(x)dx=0,\qquad\int_{\R^d} W_\beta(x,y)\bar M_\beta(y)dy=0.\]
In the case of the repulsive logarithmic kernel, the uniform bound $\sup_{N\ge1}\widehat Z_{N,\beta}<\infty$ is known to extend to all~$\beta<\infty$ by~\cite[Theorem~2.5]{2506.22083}.

To exploit the above estimates on the modified partition function $\widehat Z_{N,\beta}$, we show that it allows again to control the density of the Gibbs equilibrium $M_{N,\beta}$ with respect to the product mean-field equilibrium $M_\beta^{\otimes N}$, in the spirit of~\eqref{eq:moment-MN}. Indeed, using the identity
\[e^{-\frac\beta2|v|^2-\beta V}=Z_\beta M_\beta e^{\beta W\ast M_\beta},\]
a direct computation yields
\begin{eqnarray*}
\int_{\Dd^N}\frac{|M_{N,\beta}|^p}{|M_\beta^{\otimes N}|^{p-1}}
&=&\frac{\int_{(\R^d)^N}\exp\big(-\frac{p\beta}{2N}\sum_{j\ne l}^NW_\beta(x_j,x_l)+\frac{p\beta}{N}\sum_{j=1}^NW\ast \bar M_\beta(x_j)\big)\bar M_\beta^{\otimes N}}{\Big(\int_{(\R^d)^N}\exp\big(-\frac{\beta}{2N}\sum_{j\ne l}^NW_\beta(x_j,x_l)+\frac{\beta}{N}\sum_{j=1}^NW\ast\bar M_\beta(x_j)\big)\bar M_\beta^{\otimes N}\Big)^p}\\
&\le&e^{2p\beta \|W\ast\bar M_\beta\|_{L^\infty}}\frac{\widehat Z_{N,p\beta}}{(\widehat Z_{N,\beta})^p},
\end{eqnarray*}
which can then be controlled using the above partition function estimates.
With this control at hand, the a priori estimate of Lemma~\ref{lem:apriori-Liouv} on the Liouville solution is easily adapted and the proof of Theorem~\ref{th:main} then carry over verbatim.

\bibliographystyle{plain}
\bibliography{biblio}

\begin{thebibliography}{10}

\bibitem{Benarous-Brunaud}
G.~Ben~Arous and M.~Brunaud.
\newblock M\'ethode de {L}aplace: \'etude variationnelle des fluctuations de
  type champ moyen.
\newblock {\em Stochastic}, 31:79--144, 1990.

\bibitem{BGS-17}
T.~Bodineau, I.~Gallagher, and L.~Saint-Raymond.
\newblock From hard sphere dynamics to the {S}tokes-{F}ourier equations: an
  {$L^2$} analysis of the {B}oltzmann-{G}rad limit.
\newblock {\em Ann. PDE}, 3(1):Paper No. 2, 118, 2017.

\bibitem{BDJ-24}
M.~Bresch, D.~Duerinckx and P.-E. Jabin.
\newblock A duality method for mean-field limits with singular interactions.
\newblock Preprint, arXiv:2402.04695.

\bibitem{Brydges}
D.~Brydges.
\newblock A short course on cluster expansions.
\newblock In {\em {Critical Phenomena, Random Systems, Gauge Theories}}, Les
  Houches Lectures. 1984.

\bibitem{2506.22083}
M.~G. Delgadino and R.~S. Gvalani.
\newblock Sharp mean-field estimates for the repulsive log gas in any
  dimension.
\newblock Preprint, arXiv:2506.22083.

\bibitem{DSR-21}
M.~Duerinckx and L.~Saint-Raymond.
\newblock Lenard-{B}alescu correction to mean-field theory.
\newblock {\em Probab. Math. Phys.}, 2(1):27--69, 2021.

\bibitem{Fernandez-Procacci}
R.~Fern{\'a}ndez and A.~Procacci.
\newblock {Cluster Expansion for Abstract Polymer Models. New Bounds from an
  Old Approach}.
\newblock {\em Communications in Mathematical Physics}, 274(1):123--140, 2007.

\bibitem{Gine-Latala-Zinn-00}
E.~Gin{\'e}, R.~Lata{\l}a, and J.~Zinn.
\newblock {Exponential and Moment Inequalities for U-Statistics}.
\newblock In E.~Gin{\'e}, D.~M. Mason, and J.~A. Wellner, editors, {\em High
  Dimensional Probability II}, pages 13--38, Boston, MA, 2000. Birkh{\"a}user
  Boston.

\bibitem{Jansen}
S.~Jansen.
\newblock {Gibbsian point processes}.
\newblock Lecture notes, Winter 2017/18.

\bibitem{Petrache-Serfaty}
M.~Petrache and S.~Serfaty.
\newblock Next order asymptotics and renormalized energy for {R}iesz
  interactions.
\newblock {\em J. Inst. Math. Jussieu}, 16(3):501--569, 2017.

\bibitem{Rosenzweig-pers}
M.~Rosenzweig and S.~Serfaty.
\newblock Personal communication.

\bibitem{Rougerie-Serfaty}
N.~Rougerie and S.~Serfaty.
\newblock {Higher-Dimensional Coulomb Gases and Renormalized Energy
  Functionals}.
\newblock {\em Commun. Pur. Appl. Math.}, 69:519--605, 2016.

\end{thebibliography}

\end{document}